\documentclass[preprint,12pt,authoryear]{elsarticle}

\usepackage{amssymb}
\usepackage{amsmath,amsthm}
\usepackage{amssymb}
\usepackage{comment}
\newtheorem{theorem}{Theorem}[section]

\newtheorem{lemma}[theorem]{Lemma}
\newtheorem{proposition}[theorem]{Proposition}

\usepackage{color}

\theoremstyle{definition}
\newtheorem{definition}[theorem]{Definition}

\newtheorem{example}[theorem]{Example}
\newtheorem{fact}{Fact}

\journal{Journal of Mathematical Analysis and Applications}

\begin{document}

\begin{frontmatter}



\title{On G\^ateaux differentiability of strongly cone paraconvex vector-valued mappings}


\author{E. M. Bednarczuk and K. W. Le\'sniewski}

\address{Ewa M. Bednarczuk\\
System Research Institute\\ 
Polish Academy of Sciences\\
01-447 Warszawa, Poland\\
E-mail: Ewa.Bednarczuk@ibspan.waw.pl\\
Warsaw University of Technology\\ 
The Faculty of Mathematics and Information Science\\
00-662 Warszawa, Poland\\

Krzysztof W. Le\'sniewski\\
Warsaw University of Technology\\ 
The Faculty of Mathematics and Information Science\\
00-662 Warszawa, Poland\\
E-mail: k.lesniewski@mini.pw.edu.pl}

\begin{abstract}
We prove that a strongly cone paraconvex mapping defined on a normed space $X$ and taking values in a reflexive separable Banach space $Y$ is G\^ateaux differentiable on a dense $G_{\delta}$ subset of $X$. Our results are generalizations of 
Rolewicz's theorems (Theorem 3.1) from \cite{Rolewicz2011}.
\end{abstract}


\end{frontmatter}



\section{Introduction}

Let $\alpha:[0,+\infty)\rightarrow[0,+\infty)$ be a nondecreasing function satisfying the condition
\begin{equation}
\label{eq_alpha}
	\lim_{t\rightarrow 0^{+}}\frac{\alpha(t)}{t}=0.
\end{equation}
	Let $X$ be a normed space and $\Omega\subset X$ be an open  convex  set. A  continuous real-valued function $f:X\rightarrow\mathbb{R}$ is {\em strongly $\alpha(\cdot)$-paraconvex} on $\Omega$ with constant $C>0$ (\cite{Olech2005}), if
	$$
	f(\lambda x_{1}+(1-\lambda)x_{2})\le\lambda f(x_{1})+(1-\lambda) f(x_{2})+C\min\{\lambda,1-\lambda\}\alpha(\|x_{1}-x_{2}\|)
	$$
	for all $x_{1},x_{2}\in\Omega$, $\lambda\in[0,1]$. We say that a continuous real-valued function $f:X\rightarrow\mathbb{R}$ is {\em strongly paraconvex} on $\Omega$, if there is a function $\alpha(\cdot)$ satisfying \eqref{eq_alpha} such that $f$ is strongly $\alpha(\cdot)$-paraconvex on $\Omega$. For $\alpha\equiv0$ we get convex functions.
For $\alpha(\cdot)=\|\cdot\|^2$ we get  {\em semiconvex} functions (see \cite{cannarsa_semiconcave_2004} for application to the the analysis of the Hamilton-Jacobi equation). Characterizations and properties of semiconvex functions were studied in \cite{Tabor2012}, \cite{VanNgai2008}, \cite{Zajicek2018}. We say that $f: X\rightarrow \mathbb{R}\cup \{+\infty\}$ is a {\em approximately convex} at $x_0\in X$, if for every $\varepsilon >0$ there exists $\delta>0$ such that for all $x,y\in X$, $\|x-x_0\|<\delta$, $\|y-x_0\|<\delta$ and $\lambda \in [0,1]$
\begin{equation}
\label{eq:approx}
f(\lambda x+(1-\lambda)y)\le\lambda f(x)+(1-\lambda)f(y)+\varepsilon \lambda (1-\lambda)\|x-y\|.
\end{equation}
In \cite{Daniilidis2004b} it was proved that locally Lipschitz function are approximately convex if and only if Clarke subdifferential is a submonotone operator. We say that $f: X\rightarrow \mathbb{R}\cup \{+\infty\}$ is {\em uniformly approximate convex}, if for arbitrary $\varepsilon>0$ there is $\delta>0$ such that for $x,y $ such that $\|x-y\|<\delta$ and \eqref{eq:approx} holds. The connection between {\em uniformly approximate convex functions} (see e.g.  \cite{Zajicek2007}) and strongly paraconvex functions was established in Theorem 4 of \cite{Mig2005a}. 

 G\^ateaux and Fr\'echet differentiability properties for different classes of functions containing convex functions were studied e.g. in \cite{Rolewicz2006b} for strongly paraconvex functions, in  \cite{Tiser2015} for quasiconvex functions.  In \cite{Rolewicz2006b} Rolewicz proved the following theorem.
  \begin{theorem}[\cite{Rolewicz2006b}, Theorem 4]
    \label{th_rol_2006}
  Let $\Omega$ be an open convex subset of a separable Banach space $X$. Let $f:X\rightarrow\mathbb{R}$ be a strongly $\alpha(\cdot)$-paraconvex function on $\Omega$.
  Then there is a dense $G_{\delta}$ set $A_{G}\subset\Omega$ such that $f$ is G\^ateaux differentiable at every point of $A_{G}$.
  \end{theorem}
For Fr\'echet differentiability the following result has been proved by Rolewicz in \cite{Rolewicz2005}. 
\begin{theorem}[\cite{Rolewicz2005}, Theorem 4.6]
\label{th:frechet}
Let $\Omega$ be an open convex set in an Asplund space $X$. Let $f: X \rightarrow \mathbb{R}$ be a strongly $\alpha(\cdot)$-paraconvex function defined on $\Omega$. Then the set $A_0$ of points where $f$ is Fr\'echet differentiable is a dense $G_\delta$ set.
\end{theorem}

In the present paper we investigate G\^ateaux differentiability of strongly cone paraconvex mappings $f:X \rightarrow Y$, where $X$ is normed and $Y$ is a Banach space. In particular we provide an extension of Mazur's Theorem \cite{Mazur1933} to vector-valued mappings.

Let $X$ and $Y$ be normed spaces and $f:X\rightarrow Y$. Let $K\subset Y$ be a convex cone, i.e. $K+K\subset K, \lambda K \subset K$ for $\lambda \ge 0$ and the partial order generated by cone $K$ be defined as follows $x\le_K y \Leftrightarrow y-x\in K.$
\begin{definition}[\cite{Rolewicz2011}]
\label{def:para2}
Let $\Omega\subset X$ be an open convex set of a normed space $X$.
Let $0\neq k\in K\subset Y,$ where  $K$ is closed  and convex cone. We say that a continuous mapping  $f: X\rightarrow Y$ is {\em strongly $\alpha(\cdot)$-$k$-paraconvex} on $\Omega$, with constant $C> 0$, if for all $x,y\in \Omega$ and $0\le \lambda \le 1$ we have 
\begin{equation}
\label{para}
f(\lambda x+(1-\lambda)y)\le_{K} \lambda f(x)+(1-\lambda) f(y)+C \mbox{min}[\lambda, 1-\lambda] \alpha(\|x-y\|_{X})k.
\end{equation}
We say that $f$ defined on an open convex subset $\Omega\subset X$ and with values in $Y$ is strongly $\alpha(\cdot)$-$K$-paraconvex on $\Omega$, if it is {\em strongly $\alpha(\cdot)$-$k$-paraconvex} for all $k\in Int_{r}K,$ where $Int_{r}K$ is relative interior of $K.$ We say that $f$ is {\em strongly cone paraconvex}, if there exists a function $\alpha(\cdot)$ satisfying \eqref{eq_alpha} such that $f$ is strongly $\alpha(\cdot)$-$k$-paraconvex.
\end{definition}
By Lemma 5 of \cite{rolewicz2}, condition \eqref{para} can be equivalently rewritten as
\begin{equation}
 	\label{def:para2}
 	f(\lambda x+(1-\lambda)y)\le_{K} \lambda f(x)+(1-\lambda) f(y)+C_1 \lambda (1-\lambda) \alpha(\|x-y\|_{X})k
 \end{equation}
 for some $C_1\ge 0.$  If $\alpha\equiv 0$ we get cone-convex functions.
See \cite{Ama} for applications of vector-valued paraconvex mappings to optimality conditions in multicriteria optimization.

 We prove sufficient conditions for G\^ateaux and Fr\'echet differentiability of strongly cone paraconvex vector-valued mappings. For cone convex mappings analogous  conditions can be found in \cite{Borwein1982}. 

  In \cite{Rolewicz2011} Rolewicz proved the following theorem.
\begin{theorem}[\cite{Rolewicz2011}, Theorem 3.1]
\label{th:Rol}
Let $\Omega \subset X$ be an open convex set in Banach space $X$. Let $K$ be a convex closed pointed cone in $\mathbb{R}^n$. Let $f: X \rightarrow \mathbb{R}^n$ be a strongly $\alpha(\cdot)$-$K$-paraconvex mapping on $\Omega.$
Then the mapping $f$ is: 
{\begin{itemize}
	\item[(i)] Fr\'echet differentiable on a dense $G_\delta$ set provided $X$ is an Asplund space,
	\item[(ii)] G\^ateaux differentiable on a dense $G_\delta$ set provided $X$ is separable. 
\end{itemize}}
\end{theorem}
We prove $(ii)$  of Theorem \ref{th:Rol} for strongly cone paraconvex vector-valued mappings with values in reflexive and separable Banach spaces $Y$ partially ordered by a normal closed convex cone $K$. In this way we provide a positive answer to Problem 4 of \cite{Rolewicz2011}:
{\it Does Theorem \ref{th:Rol} hold for infinite dimensional spaces $Y$?}
We also discuss extension of (i) to infinite-dimensional spaces by following  essentially the lines of reasoning of   \cite{Borwein1982}, where cone convex mappings were investigated with values in Banach spaces $Y$ ordered by cones with bounded bases. 

The paper is organized as follows. In Section \ref{sec:introdution} we provide basic properties of strongly $\alpha(\cdot)$-$k$-paraconvex mappings. Our main result are in Section
\ref{sec:main} and Section \ref{sec:frechet}. 
\section{Preliminary facts}
\label{sec:introdution}
Let $Y$ be a real normed space with with the topological dual $Y^*$. 
Let $\alpha(\cdot)$ be a nondecreasing function  mapping the interval $[0,+\infty)$ into itself such that
\begin{equation}
\label{eq:alpha}
\lim_{t\downarrow 0} \frac{\alpha(t)}{t}=0.
\end{equation}
Let $K \subset Y$ be a closed convex cone which defines partial ordering $x\le_K y \Leftrightarrow y-x\in K.$ Let $K^*:=\{y^*\in Y^* : y^*(k)\ge 0, \forall k\in K \}$ be the dual cone. The following representation is valid in locally convex spaces $Y$ (\cite{Jahn2014}, Lemma 3.21)
\begin{equation}
\label{eq:dual}
K=\{x\in Y\ |\ y^{*}(x)\ge 0,\ \ y^{*}\in K^{*}\}.
\end{equation}

%

In the sequel we will use the following lemma.
\begin{lemma}
	\label{lemma-paraconvex} 
	Let $X$ and $Y$ be  normed spaces. 
	Let $\Omega\subset X$ be an open convex set. Let $K\subset Y$ be closed convex pointed cone, $k\in K$. The following conditions are equivalent:
	\begin{enumerate}
		\item[(i)]
		mapping $f:X\rightarrow Y$  is strongly $\alpha(\cdot)$-$k$-paraconvex on $\Omega$  with constant  $C\ge 0,$
		\item[(ii)] $\mbox{for all } y^*\in K^* \mbox{ composite functions }  y^{*}\circ f:X\rightarrow\mathbb{R}
		\mbox{ are strongly }\alpha(\cdot)$-paraconvex on $\Omega$ with constant $Cy^*(k)$.
	\end{enumerate}
\end{lemma}
\begin{proof}
Implication $(i)\Rightarrow (ii)$ follows directly from  \eqref{eq:dual}. 
Now, let us assume that, for all $y^*\in K^*$, functions $y^*\circ f$ are strongly $\alpha(\cdot)$-$k$-paraconvex with constant $Cy^*(k)$ i.e. 
\begin{align*}
&y^*(f(\lambda x_1+(1-\lambda)x_2))\\& \le \lambda y^*(f(x_1))+(1-\lambda) y^*(f(x_2))+C \mbox{min}[\lambda, 1-\lambda] \alpha(\|x_1-x_2\|)y^*(k),
\end{align*}
then by \eqref{eq:dual},  \eqref{para} holds.
\end{proof}

 The proof of Fact \ref{fact-rownowaznosc-2-paraconvex}  follows from Proposition 1.1.3, \cite{cannarsa_semiconcave_2004}, see  also Example 1 in \cite{bed-les-2017}.
\begin{fact}
\label{fact-rownowaznosc-2-paraconvex}
Let $X$ be a Hilbert space and let $Y$ be a Banach space, $K\subset Y$ closed convex cone, $k\in K$, $f:X \rightarrow Y$ and $\Omega \subset X$ be and open convex set. The following conditions are equivalent:
\begin{itemize}
\item[(i)] $f$ is strongly $2$-$k$-paraconvex ( i.e. $\alpha(t):=t^2$) on $\Omega$ with constant $C\ge0,$ 
\item[(ii)] $y^*(f(\cdot))+C \|\cdot\|^2 y^*(k)$ is convex on $\Omega$ for all $y^*\in K^*$.
\end{itemize}
\end{fact}

By Proposition 1.1.3 of \cite{cannarsa_semiconcave_2004} condition $(ii)$ of Fact \ref{fact-rownowaznosc-2-paraconvex}  is equivalent to the following representation of $y^*\circ f$: 
\begin{equation}
\label{eq:eq3}
y^*(f(x)) = u_1(x)y^*(k) + u_2(x)y^*(k) \mbox{ for all } y^*\in K^*,
\end{equation}
where $u_1$ is convex, $u_2(x)\in C^2(\Omega)$, $\|D^2 u_2\|_\infty\le 2C$ and both functions $u_1, u_2$ depends on $y^*$. 

The following example illustrates the concept of strong $2$-$k$-paraconvexity.

\begin{example}
\label{ex:3}
Let $\Omega \subset \mathbb{R}$ be an open convex bounded set.
Let $Y=c_0$, where $c_0$ is a space of all sequences convergent to zero with sup norm. Let $k=(k_1, k_2, \dots)$ be an element of nonnegative cone in $c_0$, i.e. $K=c^+_0:=\{x=(x_i) : x_i\ge 0, i=1,2,\dots\}$. Let $f:\Omega\rightarrow Y$ be a mapping defined as  $f(x)=(f_1(x)k_1, f_2(x)k_2, \dots).$  We assume that  
$$
f_i(x) = u_1^i(x) + u_2^i(x), 
$$
where $u_1^i$ is convex, $u_2^i(x)\in C^2(\Omega)$ and $\|D^2 u^i_2\|_\infty\le C$, $i=1, 2, \dots, $  i.e. $f_i$ are strongly $2$-paraconvex scalar-valued functions, see Proposition 1.1.3 of \cite{cannarsa_semiconcave_2004}. We show that  
$f$ is strongly $2$-$k$-paraconvex. By Lemma \ref{lemma-paraconvex}, it is enough to show that for all $y^*\in K^*$ composite function $y^* \circ f$ is represented as in \eqref{eq:eq3}.
Let us take $y^*\in K^*$, i.e. $y^*=(y_1, y_2, \dots)\in \ell^1_+$, we have
\begin{align*}
(y^*\circ f)(x)=\sum\limits_{i=1}^\infty y_i k_i  f_i(x)= \underbrace{\sum\limits_{i=1}^\infty y_i k_i u^i_1(x)}_{\mbox{is convex }} + \sum\limits_{i=1}^\infty y_i k_i u_2^i(x).  
\end{align*}
Now let us prove that $\bar{u}_{2}(x):=\sum\limits_{i=1}^\infty y_i k_i u_2^i(x)$ satisfies the following conditions:  $\bar{u}_{2}(x)\in C^2(\Omega)$, $\|D^2 \bar{u}_{2}\|_\infty\le \bar{C} $ for some $\bar{C}\ge0$. Since the  second derivative of $u^i_2$ is bounded, from the Lagrange theorem, we get that the first derivative of $u^i_2$ is bounded on $\Omega$, $i=1,2,\dots$.  Furthermore, we get that $u^i_2$ is bounded on $\Omega$, since it is uniformly continuous on $\Omega$.  
  Since $k\in c_0^+$ we have $|u_2^i(x)k_iy_i|\le M_1y_i$, $i=1, 2, \dots$ for some $M_1\ge0$ and $|\frac{du_2^i(x)}{dx}k_iy_i|\le M_2$ for some $M_2\ge 0$. 
  By applying two times  the Weierstrass test we obtain
$$
\left|\frac{d^2}{dx^2}\sum\limits_{k=1}^\infty y_i k_i u_2^i(x)\right|\le M \sum\limits_{i=1}^\infty y_i=\bar{C}< \infty \ \mbox{ for some } M\ge 0.
$$
We get that $\bar{u}_2=\sum\limits_{i=1}^\infty y_i u_2^i(x)$ is twice differentiable and $\|D^2 \bar{u}_2\|_\infty\le \bar{C}$. 

\end{example}

We say that Banach space $Y$ is weakly sequentially complete if every weak Cauchy sequence is  weakly convergent. As an example we can take all reflexive spaces or $\ell_1$. 

In a normed space $Y$, a convex cone $K \subset Y$ is  normal, if there is $\gamma>0$ such that, if $0\le_K x \le_K y$, then $\|x\|\le \gamma \|y\|$ for all $x, y \in Y.$ 

Let $\Omega$ be an open set in normed space $X$ and let $x_0\in \Omega$. We say that direction $h\in X$ is admissible, if $x_0+th\in \Omega$ for $t$ sufficiently small. 

In \cite{bed-les-2017} we proved the following propositions.
\begin{proposition}[\cite{bed-les-2017}, Proposition 1]
\label{stwierdzenie-monotonicznosc-a-ilorazu}
	Let  $X$ be a normed space and let    $Y$ be a Banach space.  Let cone  $K\subset Y$ be closed and convex. Let  $f:X\rightarrow Y$ be strongly  $\alpha(\cdot)$-$k$-paraconvex on an open convex set  $\Omega\subset X$ with constant   $C\ge 0$ and  $k\in K\setminus\{0\}$. Then for all  $x_{0}\in \Omega$ and admissible direction $ h\in X$  such that
	$\|h\|=1$, mapping  $\phi:\mathbb{R}\rightarrow Y$ defined as
	\begin{equation}
	\label{phi0}
	\phi(t):=\frac{f(x_{0}+th)-f(x_{0}+t_{0}h)}{t-t_{0}}+
	C\frac{\alpha(t-t_{0})}{t-t_{0}}k \ \text{for}\ t_{0}<t, 
	\end{equation}
	where $t_{0}\in\mathbb{R}$	
	is  $\alpha(\cdot)$-nondecreasing, i.e.
	\begin{equation}
	\label{nier}
	\phi(t)-\phi(t_{1})+C\frac{\alpha(t_1-t_{0})}{t_{1}-t_{0}}k\in K\ \ \text{for}\ \ t_{0}<t_{1}<t.
	\end{equation}
	
	\end{proposition}

\begin{proposition}[\cite{bed-les-2017}, Proposition 2]
\label{stwierdznie-ograniczonosc-a-ilorazu}
Let  $X$ be a normed space and let    $Y$ be a Banach space.  Let cone  $K\subset Y$ be closed and convex. Let  $f:X\rightarrow Y$ be strongly  $\alpha(\cdot)$-$k$-paraconvex on a convex set  $\Omega\subset X$ with constant   $C\ge 0$ and  $k\in K\setminus\{0\}$. Then for all  $x_{0}\in \Omega$ admissible direction $ h\in X$  such that
	$\|h\|=1$, mapping  $\phi:\mathbb{R}\rightarrow Y$ defined as
	\begin{equation*}
	\label{phi1}
	\phi(t):=\frac{f(x_{0}+th)-f(x_{0})}{t}+
	C\frac{\alpha(t)}{t}k,
	\end{equation*}
	where $t>0$	
	is locally $K$-bounded from below, i.e. there are $a\in Y$ and $\delta >0$ such that
	\begin{equation}
	\label{ograniczenie1}
	\phi(t)-a\in K\ \  (i.e. \ \ \phi(t)\ge_K a)\ \ \text{for}\ \ 0<t<\delta.
	\end{equation}
	\end{proposition}
Basing ourselves on the above propositions we get the following result.

\begin{theorem}[Bednarczuk, Le\'sniewski \cite{bed-les-2017}, Theorem 2]
	\label{val}
	Let  $X$ be a normed space. Let $Y$ be sequentially complete Banach space. Let  $K\subset Y$ be closed convex normal cone, $k\in K$. If  $f :X \rightarrow Y$ is strongly  $\alpha(\cdot)$-$k$-paraconvex on an open convex set  $\Omega$, then $f$ has directional derivative at any $x_0\in \Omega$ in every admissible direction $h$, i.e. the strong limit exists
	$$
	f'(x_{0},h)=\lim_{t\rightarrow 0^+}\frac{f(x_{0}+th)-f(x_{0})}{t},
	$$
	where $h$ is an admissible direction at $x_0$ if $x_0+th\in \Omega$ for $t>0$ sufficiently small.
\end{theorem}

\section{G\^ateaux differentiability}
\label{sec:main}
In this section we  prove  G\^ateaux differentiability for strongly $\alpha(\cdot)$-$k$-paraconvex mappings with values in reflexive and separable Banach space $Y.$

\begin{definition}
\label{definicja-pochodna-gataux}
Mapping $f:X\rightarrow Y $ is {\bf G\^ateaux differentiable} at  $x_0$, if the strong limit
\begin{equation}
\label{gateax}
f'(x_0;h):= \lim\limits_{t\rightarrow 0} \frac{f(x_0+th)-f(x_0)}{t}
\end{equation} 
exists for every direction $h\in X$ and the mapping $h\rightarrow  f'(x_0;h)$ is linear and continuous with respect to $h.$ 
\end{definition}
Even for cone-convex mappings, the existence  of directional derivatives does not imply G\^ateaux differentaibility, \cite{borwein1986}. It is easy to observe that, if the limit \eqref{gateax} exists, then directional derivative of strongly cone paraconvex vector-valued mappings is sublinear and positively homogeneous with respect to $h$. 
\begin{fact}
\label{fact:sublinearity}
Let $X$ be a normed space. Let  $Y$ be a Banach space and let $K\subset Y$ be closed convex cone, $k\in K$.
 Let  $f:X \rightarrow Y$ be strongly $\alpha(\cdot)$-$k$-paraconvex  on a convex set $\Omega$ and let  $f'(x_0;h)$ exists at  $x_0$ in every admissible direction  $h$ at $x_0$. Then $f'(x_0;\cdot)$ is sublinear and positively homogeneous. 
\end{fact}
\begin{proof}
We will prove only sublinearity. 
Let $x_0\in \Omega$. Let $h_1, h_2\in X$ be  admissible directions at $x_0$ such that $h_1\neq h_2$. From \eqref{para}, for $\lambda=\frac 1 2$, $x_1= \frac 1 2 x_0 + 2th_1$, $x_2=\frac 1 2 x_0+2th_2$ and $t>0$ sufficiently small, we get 
 \begin{align*}
&\frac{f(\frac 1 2 (x_0 + 2th_1)+ \frac 1 2 (x_0 + 2th_1))}{t}\\&\le_K  \frac{f (x_0 + 2th_1)}{2t}+ \frac{f (x_0 + 2th_2)}{2t} + C\frac{\alpha(2t\|h_1-h_2\|)}{2t}k
  \end{align*}

  The fact that $\frac{\alpha(2t\|h_1-h_2\|)}{2t}\rightarrow 0$, when $t\rightarrow 0^+$ completes the proof.
 \end{proof}
 \begin{fact}
 \label{fakt_nabla}
 Let $X$ be a normed space.  Let  $Y$ be a Banach space and let $K\subset Y$ be closed convex cone, $k\in K$.
 Let  $f:X \rightarrow Y$ be strongly $\alpha(\cdot)$-$k$-paraconvex  (on $X$) and let directional derivative $f'(x_0;h)$ exists at  $x_0\in X$ in every direction  $h\in X.$ Then 
 \begin{equation}
 \label{nierownosc_pochodna_nabla}
 f'(x_0;h)\le_K \frac{f(x_0+th)-f(x_0)}{t}+C\frac{\alpha(t)}{t}k \mbox{ for all } t>0.
 \end{equation}
 \end{fact}
 \begin{proof}
Formula \eqref{nier} can be equivalently rewritten as
\begin{align}
\begin{aligned}
\label{nierrr}
\frac{f(x_0+th)-f(x_0+t_0h)}{t-t_0} - \frac{f(x_0+t_1h)-f(x_0+t_0h)}{t_1-t_0}\\
+ C \frac{\alpha(t-t_0)}{t-t_0}k \in K, \mbox{ for all } t_0<t_1<t.
\end{aligned}
\end{align} 
In view of the closedness of cone $K$, by letting $t_1\downarrow t_0$ in \eqref{nierrr} and keeping $t>t_0$ fixed, we get 
 $$f'(x_0+t_0h;h) \le_K  \frac{f(x_0+th)-f(x_0+t_0h)}{t-t_0} +C \frac{\alpha(t-t_0)}{t-t_0}k \mbox{ for }t>t_0.$$

 For $t_0=0$ this gives the conclusion.
 \end{proof}
 \begin{definition}
 Let $K\subset Y$ be a closed convex cone.
We say that $f: X \rightarrow Y$ is locally vector-bounded (w.r. to cone $K$) on an open closed convex set $\Omega\subset X$, if for all $x_0\in \Omega$ there is an open set $U\ni x_0$ and $k\in K$ such that 
$$-k \le_K f(x)\le_K k \mbox{ for all } x\in U.$$
 \end{definition}
 
 \begin{proposition}[\cite{Rolewicz2012}, Proposition 3.3]
 \label{prop:lipschitz-Rolewicz}
Let $X$ be a normed space. Let a mapping $f$ defined on an open convex subset $\Omega\subset X$ with values in the Banach space $Y$ ordered by a convex pointed cone $K$ be locally strongly $\alpha(\cdot)$-$k$-paraconvex on $\Omega$ (for every $x_0\in \Omega$, $f$ is strongly $\alpha(\cdot)$-$k$-paraconvex on some open set $U\ni x_0$) and locally vector-bounded. Then $f$ is locally vector Lipschitz, i.e.
\begin{equation}
\label{eq:locally-vlipschitz}
\exists L, \exists U, \forall x, u \in U, \  -L\|u-x\|k \le_K f(u)-f(x) \le_K L\|u-x\|k   
\end{equation}
and, if, moreover, $K$ is normal, then 
\begin{equation}
\label{eq:llipschitz}
\exists L, \exists U, \forall x, u \in U, \  -L\|u-x\| \le \|f(u)-f(x)\| \le L\|u-x\|  
\end{equation}

 \end{proposition}

 \begin{proposition}
 \label{stwciaglosc}
 Let $X$ be a normed space and let $Y$ be a Banach space.
 Let  $f:X \rightarrow Y$ be strongly $\alpha(\cdot)$-$k$-paraconvex (on $X$) and locally bounded at  $x_0$, where $K\subset Y$  is closed convex and normal cone.
 If directional derivative   $f'(x_0;h)$ exists and  is linear with respect to  $h$ for all $h\in X$, then  
  $f'(x_0;\cdot)$ is continuous with respect to  $h$ on $X$.
   \end{proposition}
   \begin{proof}
By Proposition \ref{prop:lipschitz-Rolewicz} there exist
     $L>0$ and $\delta >0$ such that
   $$
 \|f(x_0+t(h-h_0)-f(x_0)\|\le L t\|h-h_0\| \mbox{ for } 0<t<\delta, h\in X.
   $$ 
   Let us fix $\varepsilon > 0 $. For $t>0$ sufficiently small we have
   $$
   \|C\frac{\alpha(t)}{t}k\|\le \frac{\varepsilon}{2\gamma}
   ,$$ where $\gamma$ is a constant from the definition of normal cone. 
Let us take $\delta = \frac{\varepsilon}{2L\gamma}$. 
   From \eqref{nierownosc_pochodna_nabla} and from the fact that $K$ is normal 
   \begin{align*}
   \label{eq-pochodna-nier1}
   &\|f'(x_0;h-h_0)\| \le \gamma\|\frac{f(x_0+th)-f(x_0)}{t} + C\frac{\alpha(t)}{t}k\|\\&\le \gamma L \|h-h_0\| + \gamma \|C\frac{\alpha(t)}{t}k\|\le \gamma L \frac{\varepsilon}{2L\gamma} +\frac{\varepsilon}{2\gamma}\gamma=\varepsilon.   
   \end{align*}

   \end{proof}
	Now we are ready to prove the G\^ateaux differentiability of strongly cone paraconvex mappings. The following result holds.
\begin{theorem}
\label{gateaux1}
	Let $X$ be a separable Banach space. Let $Y$ be reflexive and separable Banach space. Let $K$ be closed convex and normal cone in $Y,$ $k\in K$. Let $f:X\rightarrow Y$ be  strongly $\alpha(\cdot)$-$k$-paraconvex (on $X$) and locally vector-bounded. Then  $f$  is G\^ateaux differentiable on some dense  $G_\delta$ set.
	
	\end{theorem}

	\begin{proof}
	 We will show that the limit $f'(x_0;h)$ exists and is linear (with respect to $h$) for all $x_0\in A_0,h\in X,$ where $A_0\subset X$ is dense $G_\delta$ set.
	Continuity of directional derivative will follow from Proposition  \ref{stwciaglosc}.

	From the fact that $Y$ is reflexive and separable, $Y^*$ is separable (\cite{Kreyszig2002}, Theorem 4.6-8).
	We also know that any subset of a separable Banach space is separable.  Let  $\{\ell^{*}_{i}\}_{i\in\mathbb{N}}\subset K^*$ be a dense subset. From Lemma \ref{lemma-paraconvex}  every function   $\ell^*_i \circ f$, 
$i\in \mathbb{N}$, is strongly $\alpha(\cdot)$-paraconvex. From Rolewicz Theorem  (Theorem  \ref{th_rol_2006}) every function $\ell_{i}^*\circ f$ is G\^ateaux differentiable on a dense set $A_{i}$, which is $G_\delta$, $i\in \mathbb{N}$.

Let $A_{0}:=\bigcap_{i\in \mathbb{N}} A_{i}$. From Baire's Theorem, $A_{0}$ is $G_{\delta}$.

Let us take  $y^{*}\in K^{*}.$
We have $y^{*}=\lim_{m\rightarrow\infty}\ell_{i_{m}}^*$ and
$$
y^{*}(f(x))=\lim_{m\rightarrow\infty} \ell^*_{i_{m}}(f(x))\ \ \text{ for every  } x\in X.
$$
Let us fix $x_{0}\in A_{0}$. For every direction $h\in X$ and $t>0$ we have
$$
\begin{array}{l}
y^{*}(f(x_{0}+th))=\lim\limits_{m\rightarrow\infty}\ell^{*}_{i_{m}}(f(x_{0}+th)),\\
y^{*}(f(x_{0}))=\lim\limits_{m\rightarrow\infty}\ell^{*}_{i_{m}}(f(x_{0})).
\end{array}
$$
For $t> 0$ we have
\begin{equation}
\label{eq_directional}
\frac{y^{*}(f(x_{0}+th))-y^{*}(f(x_{0}))}{t}=
\lim_{m\rightarrow\infty}\frac{\ell^{*}_{i_{m}}(f(x_{0}+th))-\ell^{*}_{i_{m}}(f(x_{0}))}{t}.
\end{equation}
From the fact that $x_{0}\in A_{0}$ the directional derivative i.e.  
$$
(\ell^{*}_{m}\circ f)'(x_{0};h)=\lim_{t{\rightarrow} 0^+}\frac{\ell^{*}_{i_{m}}(f(x_{0}+th))-\ell^{*}_{i_{m}}(f(x_{0}))}{t}
$$
exists for $m\in \mathbb{N}$. 
Let us show that sequence of functions
$$
g_{m}(t):=	\frac{\ell^*_{i_{m}}(f(x_{0}+th))-\ell^*_{i_{m}}(f(x_{0}))}{t}
$$
converges uniformly to
$$
g(t):=	\frac{y^{*}(f(x_{0}+th))-y^{*}(f(x_{0}))}{t}
$$
on  $(0,\delta)$ for some $\delta>0$. Let us fix  $\varepsilon>0$. There is  $N>0$ such that  
$$
\|\ell^{*}_{i_{m}}-y^{*}\|<\varepsilon\ \ \text{ for } m>N.
$$
For  $m>N$ we get
\begin{equation}
\label{eq_uniform}
|g_m(t)-g(t)|\le\|\ell^{*}_{i_{m}}-y^{*}\|\|\frac{f(x_{0}+th)-f(x_{0})}{t}\|.
\end{equation}
 From Theorem  \ref{val} we get the  existence of the  limit  $\lim\limits_{t\rightarrow 0^+}\frac{f(x_{0}+th)-f(x_{0})}{t}$ i.e. there is   $\delta>0$ such that for  $0<t<\delta$ we have
 $
 \|\frac{f(x_{0}+th)-f(x_{0})}{t}-f'(x_0;h)\|\le 1.
 $

In consequence, from inequality \eqref{eq_uniform},
there is   $N>0$ such that for  $m>N$ and for  $0<t<\delta$ we have
\begin{equation}
\label{eq_uniform_1}
|g_m(t)-g(t)|\le\varepsilon M \mbox{ for some } M>0.
\end{equation}
Passing to the limit in   \eqref{eq_directional} we get
\begin{equation}
\label{Gateaux}
(y^{*}\circ f)'(x_{0};h)=\lim\limits_{m\rightarrow\infty}
(\ell^{*}_{i_{m}}\circ f)'(x_{0};h).
\end{equation}
The above limit exists for every  $y^*\in K^*$.
Again, from Theorem \ref{val} we obtain the existence of 
 $ f'(x_0;h) = \lim\limits_{t\rightarrow 0^+}\frac{f(x_{0}+th)-f(x_{0})}{t}$ for every direction  $h$, $x_0\in X$, and hence we have the existence of $-f'(x_0;-h)=\lim\limits_{t\rightarrow 0^-}\frac{f(x_{0}+th)-f(x_{0})}{t}$. Now, it is enough to show that both of above limits are equal. Let us show  
\begin{equation}
\label{pochodna-kier-1}
f'(x_0;h)+ f'(x_0;-h)=0 \ \mbox{ for } x_0\in A_0.
\end{equation}
From sublinearity (Fact \ref{fact:sublinearity}) of direction derivative we have
$$
0=f'(x_0; h+ (-h)) \le  f'(x_0; h) +  f'(x_0; -h).$$
Let us assume that  $  f'(x_0; h) +  f'(x_0; -h) \in K\setminus\{0\}.$ Then there is   $y^* \in K^*$ such that 
\begin{equation}
\label{eq-nier-pochodna-1}
y^*(f'(x_0; h)) + y^*(f'(x_; -h)) >0.
\end{equation}
From the fact that in  \eqref{Gateaux} we have uniform convergence and every function $\ell^{*}_{i_{m}}\circ f$, $m\in \mathbb{N}$ is  G\^ateaux differentiable in $A_0$ we have  
\begin{align*}
&y^*(f'(x_0; h)) + y^*(f'(x_; -h))\\=&\lim_{m\rightarrow\infty}((\ell^{*}_{i_{m}}\circ f)'(x_{0};h)+(\ell^{*}_{i_{m}}\circ f)'(x_{0};-h))=0,
\end{align*}
which is a contradiction with  \eqref{eq-nier-pochodna-1}.
Directional derivative  $f'(x_0; h)$ is linear with respect to  $h\in X$ for $x_0\in A_0.$ From Proposition  \ref{stwciaglosc} we get the continuity of  $f'(x_0;h)$ with respect to   $h,$  which completes the proof.
\end{proof}	
\section{Fr\'echet differentiability}
\label{sec:frechet}
In this section we prove  Fr\'echet differentiability for strongly $\alpha(\cdot)$-$k$-paraconvex mappings with values in reflexive and separable Banach space $Y$, where $k$ is an element of  closed convex cone $K$ with bounded base.

We say that a convex set $B\subset K$ is a {\em base} for a convex cone $K\subset Y$, if 
$K=\bigcup\{\lambda b : \lambda \ge 0, b\in B\}$ and $0\notin cl B.$ If $K$ has a bounded base, then $K^*$ has norm interior. For example, the nonnegative orthant of $\ell_1$ i.e. cone $\ell_1^+:=\{x=(x_i)\in \ell_1: x_i \ge 0, i=1,2,\dots\}$ has a bounded base (\cite{casini_cones_2010}, Example 4.2). Some other examples can be found in  \cite{casini_cones_2010} and the references therein. 
The proof of the following theorem is based on the ideas of the proof of Theorem 5.2 of  \cite{Borwein1982}.
\begin{theorem}
\label{th:Frechet}
Let $X$ be an separable Asplund space and let $Y$ be reflexive and separable Banach space. Let $K\subset Y$ be a closed convex and let $K$ has a bounded base, $k\in K$. Let $f: X\rightarrow Y$ be continuous and strongly $\alpha(\cdot)$-$k$-paraconvex (on $X$). Then $f$ is Fr\'echet differentiable on a dense $G_\delta$ set. 
\end{theorem} 

\begin{proof}
From the fact that $K$ has a bounded base $B$ there exists a positive functional $y^*\in K^*$ such that $\inf\{y^*(b): b\in B\}\ge 1.$ From Theorem \ref{th:frechet} we have dense $G_\delta$ set $A_0$ such that $f$ is G\^ateaux differentiable on $A_0$. Since $y^*\circ f$ is strongly $\alpha(\cdot)$-paraconvex we have 
\begin{equation}
	\label{eq:eq4.1}
	\lim_{t\rightarrow 0^+}y^*\left( \frac{f(x_0+th)-f(x_0)}{t}	\right) =  (y^*\circ f)'(x_0;h)=y^*(f'(x_0;h)) 
\end{equation}
for all $x_0\in A_0, h\in X.$
From Theorem \ref{th:Rol} (i),
 for each $\varepsilon >0$ there is $\delta >0$ such that 

\begin{equation}
\label{eq:eq6}
	y^*\left( \frac{f(x_0+th)-f(x_0)}{t}- f'(x_0;h)	\right) \le \varepsilon \mbox{ for all} \|h\|\le 1, 0<t<\delta. 
\end{equation}
From \eqref{nierownosc_pochodna_nabla} we have $\frac{f(x_0+th)-f(x_0)}{t}- f'(x_0;h) \in K$ for $t>0$ and $\frac{f(x_0+th)-f(x_0)}{t}- f'(x_0;h)=\lambda_t b_t$, where $\lambda_t\ge 0$, $b_t\in B.$ Inequality \eqref{eq:eq6} is equivalent to 
$$
y^*\left(\frac{f(x_0+th)-f(x_0)}{t}- f'(x_0;h)\right)=\lambda_t y^*(b_t)<\varepsilon.
$$ 
Since $B$ is bounded and $y^*$ is bounded away from zero on $B$, $\lambda_t\le \varepsilon$ and $\|b_t\|\le M$ for some $M\ge 0$. Consequently, for each $\varepsilon>0$ there is $\delta>0$ and $V\ni 0$ such that 
$$
\frac{f(x_0+th)-f(x_0)}{t}- f'(x_0;h)\in V \mbox{ for all } \|h\|\le 1, 0<t<\delta,
$$
which completes the proof.
\end{proof}
\section{Conclusions}

From Theorem \ref{gateaux1} and Theorem \ref{th:Frechet} we get the following generalization of Rolewicz Theorem \ref{th:Rol}.
\begin{theorem}
\label{th:podsumowanie}
Let $Y$ be reflexive and separable Banach space. Let $K\subset Y$ be a closed convex and normal.
If $f: X \rightarrow Y$ is strongly $\alpha(\cdot)$-$k$-paraconvex defined on a convex set contained in a Banach space $X$, then $f$ is:
{\begin{itemize}
	\item[(i)] Fr\'echet differentiable on a dense $G_\delta$ set provided $X$ is a separable Asplund space, $K$ has a bounded base,
	\item[(ii)] G\^ateaux differentiable on a dense $G_\delta$ set provided $X$ is separable. 
\end{itemize}}
\end{theorem}


  \bibliographystyle{elsarticle-harv}
  \bibliography{library.bib}




\end{document}